\documentclass[11pt,a4paper]{article}
\usepackage[utf8]{inputenc}
\usepackage[T1]{fontenc}
\usepackage[english]{babel}
\usepackage{lmodern}
\usepackage{amsmath,amssymb,amsthm,mathtools}
\usepackage{geometry}
\usepackage{enumitem}
\usepackage[hypertexnames=false,hidelinks]{hyperref}
\newtheorem{theorem}{Theorem}[section]
\newtheorem{lemma}[theorem]{Lemma}
\newtheorem{proposition}[theorem]{Proposition}
\newtheorem{corollary}[theorem]{Corollary}
\theoremstyle{definition}

\theoremstyle{remark}
\newtheorem{remark}[theorem]{Remark}
\newtheorem{example}[theorem]{Example}
\newtheorem*{introremark}{Remark}
\newtheorem*{introproposition}{Proposition}
\newtheorem*{introdefinition}{Definition}
\newtheorem*{maintheorem}{Main Theorem}

\newcommand{\E}{\mathbb E}
\newcommand{\R}{\mathbb R}
\newcommand{\F}{\mathbb F}
\newcommand{\D}{\mathcal D}

\newcommand{\1}{\mathbf 1}
\newcommand{\Ker}{\operatorname{Ker}}

\newcommand{\rank}{\operatorname{rank}}

\title{\textbf{A noncommutative Kalman condition for\\
null controllability of backward stochastic parabolic systems}}
\author{Mohamed Fadili\\[0.35em]
\small Laboratoire de Mathématiques, Modélisation et Systèmes Automatiques,\\
\small École Normale Supérieure, Cadi Ayyad University,\\
\small BP 2400, Marrakesh, Morocco\\
\small \texttt{m.fadili@uca.ma}\\
\small \href{https://orcid.org/0000-0003-4285-7793}
{ORCID: 0000-0003-4285-7793}}
\date{}

\begin{document}
\maketitle

\begin{abstract}
We study null controllability of backward stochastic heat systems with
constant couplings in both the state and the martingale integrand, common
scalar diffusion, and a localized drift control. We prove that null
controllability is equivalent to a noncommutative word-rank condition on the
two coupling matrices and the control matrix. The proof combines a
factorization of the dual dynamics, positivity and small-time coercivity of
a finite-dimensional stochastic Gramian, and a Lebeau--Robbiano spectral
argument. Failure of the rank condition yields an invariant unobservable
subspace.
\end{abstract}

\textbf{Keywords.} Backward stochastic parabolic equation; null
controllability; observability; noncommutative Kalman condition; stochastic
Gramian; Lebeau--Robbiano method.

\textbf{MSC 2020.} 93B05, 93B07, 93E20, 35K10.

\section*{Introduction}

Controllability of coupled parabolic systems by fewer controls than state
components is governed, in the deterministic constant-coefficient setting,
by Kalman-type algebraic conditions. For stochastic parabolic systems, the
martingale integrand is an additional unknown and a coupling through that
variable produces mechanisms which are absent from deterministic systems.

In earlier work, the author proved null controllability for a backward
stochastic cascade system in which observation is transferred by a
nonvanishing subdiagonal coupling in the state equation~\cite{FadiliCascade}.
A subsequent two-equation result with Maniar showed that the coupling of the
martingale integrands can also transfer observation under a signed local
nondegeneracy condition, even for space--time dependent
coefficients~\cite{FadiliTwoState}. Here we place these two
mechanisms into a common framework for arbitrary constant matrices. In this
setting, the earlier sufficient structural assumptions are replaced by an
exact algebraic characterization.

Existing results based on the Lebeau--Robbiano method also address forward
controlled stochastic coupled systems~\cite{Liu2014,LiuLiu}. Their dual
equations are backward and contain an additional unknown martingale
integrand. This is structurally different from the problem considered here:
the controlled system is backward, while its dual is a forward equation
with the explicit multiplicative-noise coefficient \(-A_2^*z\). Thus the
negative conclusions concerning classical Kalman conditions in some
forward controlled settings do not apply directly to this system.

Recent Carleman approaches treat broad classes of forward and backward
stochastic parabolic equations with first- and zero-order
couplings~\cite{BarounEtAl2025}, and cascade systems of backward equations with one
distributed control~\cite{BouliteCascade2025}. Two coupled backward
fourth-order equations with one drift control have also been treated by a
dedicated global Carleman estimate~\cite{Wang2024}. These results allow
space-dependent lower-order coefficients and, in the cascade case, general
parabolic principal parts. Our assumptions are narrower analytically
(constant matrices and a common scalar diffusion), but they allow arbitrary
noncascade couplings in both \(y\) and \(Y\), and lead to an exact algebraic
characterization rather than a structural cascade hypothesis.

The relevant finite-dimensional condition is not the classical Kalman rank
condition associated with a single matrix. It involves all words in the
drift and martingale-coupling matrices. Equivalently, it asks the columns of
the control matrix to generate the state space as a module over the
associative algebra jointly generated by these two matrices.

For finite-dimensional linear backward stochastic equations, this
word-rank criterion originates in the exact-controllability theory of
Peng~\cite{Peng1994} and is presented in a systematic form by Lü and
Zhang~\cite{LuZhang}. Gramian and Kalman criteria have also been developed
for finite-dimensional mean-field stochastic systems~\cite{Yu2021}.
Those results concern exact controllability at the initial time of a
finite-dimensional BSDE. We do not invoke them as a controllability theorem
for this PDE. Their role is to identify the algebraic object suggested
by the internal dynamics; below we rederive both the kernel characterization
and the quantitative finite-dimensional Gramian estimate needed here.

The parabolic problem has two further difficulties. A mode-by-mode control
acts globally in space and does not produce one control supported in \(G\),
and qualitative controllability of each mode gives no estimate uniform in
the eigenvalue. Spatial localization requires the spectral inequality, while
the passage from finite spectral spaces to the full equation requires a
balance between low-frequency control cost and high-frequency heat
dissipation. Thus the finite-dimensional criterion alone does not settle the
PDE. We prove that the same algebraic condition is necessary and
sufficient after localization and supply the uniform estimates needed for
the spectral iteration.

Our proof does not seek a simultaneous cascade form, which need not exist
for two noncommuting matrices. Instead, the dual equation factorizes into a
finite-dimensional stochastic flow and the scalar heat semigroup. This
leads to a stochastic internal Gramian whose kernel is exactly the
unobservable module. Its small-time coercivity is polynomial, which is the
quantitative input required by the Lebeau--Robbiano method.

An alternative route would start from a componentwise vectorial Carleman
estimate and try to propagate the localized observation successively from
\(B^*z\) to \((w(A_1,A_2)B)^*z\) for all required words. Drift letters can
in principle be generated by localized integrations by parts, whereas
noise letters require quadratic-variation identities. For arbitrary
noncommuting words, closing such a propagation lemma entails substantial
commutator and localization bookkeeping. The common scalar diffusion
assumption offers a more direct mechanism: factorization makes the internal
Gramian independent of the Laplace eigenvalue, and the only modal
dependence is the explicit scalar factor \(e^{-d\lambda_k t}\). Its
polynomial small-time coercivity supplies the uniform modal estimate
required by the spectral argument. We therefore use the
Lebeau--Robbiano method here and leave the Carleman approach for future work.
Recent abstract spectral and telegraph-series methods establish
observability for classes of backward stochastic evolution
equations~\cite{LiuWuYangZhong2026}. Their observed evolution is backward and their
applications concern scalar degenerate, fourth-order, and heat equations.
Here the dual evolution is forward, and the additional issue is
to identify exactly when one matrix observation sees every component of a
coupled system.

The spectral iteration itself is standard. The new points are the necessary
and sufficient noncommutative word condition, its identification with the
kernel of the internal stochastic Gramian, and the quantitative estimate
that converts this finite-dimensional condition into localized parabolic
observability.

\subsection*{Setting and statement of the main theorem}

Let \(T>0\), let \(\D\subset\R^N\) be a bounded connected domain with
\(C^2\) boundary, and let \(G\subset\D\) be a nonempty open subset. Set
\[
Q=(0,T)\times\D,\qquad \Sigma=(0,T)\times\partial\D.
\]
Let \((\Omega,\mathcal F,\F,\mathbb P)\) carry a one-dimensional Brownian
motion \(W\), and assume that \(\F\) is its augmented natural filtration.
For a separable Hilbert space \(H\), set
\[
L^2_{\mathcal F_t}(\Omega;H)
=\{X:\Omega\to H:\ X\text{ is }\mathcal F_t\text{-measurable and }
\E\|X\|_H^2<\infty\}.
\]
We denote by \(L^2_\F(a,b;H)\) the Hilbert space of progressively measurable
processes \(u:(a,b)\times\Omega\to H\) such that
\[
\|u\|_{L^2_\F(a,b;H)}^2
=\E\int_a^b\|u(t)\|_H^2\,dt<\infty.
\]
Likewise, \(L^2_\F(\Omega;C([a,b];H))\) consists of adapted processes with
continuous paths and finite norm
\(\E\sup_{a\le t\le b}\|u(t)\|_H^2\). Processes are identified up to the
usual almost-everywhere equivalence. Since \(\F\) is the augmented natural
filtration of \(W\), \(\mathcal F_0\) is trivial modulo null sets and
\(L^2_{\mathcal F_0}(\Omega;H)\) is canonically identified with \(H\).
We consider
\begin{equation}
\begin{cases}
dy=\bigl[-d\Delta y+A_1y+A_2Y+B\1_Gv\bigr]dt+Y\,dW_t
       &\text{in }Q,\\
y=0&\text{on }\Sigma,\\
y(T)=y_T&\text{in }\D,
\end{cases}
\tag{I.1}
\end{equation}
where \(d>0\),
\[
A_1,A_2\in\R^{n\times n},\qquad B\in\R^{n\times m},
\]
are constant deterministic matrices.

\begin{introremark}
\textbf{Geometric input.}
Beyond standard variational well-posedness, the only geometric ingredient
used below is the Dirichlet spectral inequality
\[
\|f\|_{L^2(\D)}^2
\le C e^{C\sqrt\Lambda}\|f\|_{L^2(G)}^2,
\qquad f\in E_\Lambda.
\]
The bounded connected \(C^2\) setting is a standard framework for this
inequality. More generally, the proof applies verbatim to any
self-adjoint scalar diffusion with compact resolvent for which the same
spectral inequality and the corresponding high-frequency dissipation hold.
This observation does not cover a nonscalar matrix diffusion, which fails
the factorization established below.
\end{introremark}

\begin{introdefinition}
System (I.1) is null controllable at time \(T\) if, for every
\[
y_T\in L^2_{\mathcal F_T}(\Omega;L^2(\D)^n),
\]
there exists
\[
v\in L^2_{\F}(0,T;L^2(G)^m)
\]
such that the corresponding solution satisfies \(y(0)=0\) in
\(L^2(\D)^n\), almost surely.
\end{introdefinition}

Let \(\mathcal W(A_1,A_2)\) be the family of all finite words in
\(A_1,A_2\), including the empty word \(I_n\), and define
\begin{equation}
\mathcal R(A_1,A_2,B)
=\operatorname{span}\{w(A_1,A_2)B\eta:
w\in\mathcal W(A_1,A_2),\ \eta\in\R^m\}.
\tag{I.4}
\end{equation}

\begin{maintheorem}
For every \(T>0\), system (I.1) is null controllable if and only if
\begin{equation}
\mathcal R(A_1,A_2,B)=\R^n.
\tag{I.5}
\end{equation}
Moreover, if (I.5) holds, the control can be chosen so that
\begin{equation}
\E\int_0^T\|v(t)\|_{L^2(G)^m}^2\,dt
\le C_T\E\|y_T\|_{L^2(\D)^n}^2.
\tag{I.6}
\end{equation}
For \(0<T\le1\), the constant may be chosen so that
\begin{equation}
C_T\le C\exp\!\left(\frac CT\right).
\tag{I.7}
\end{equation}
\end{maintheorem}

\begin{introremark}
Because the algebra generated by \(A_1,A_2\) is finite-dimensional,
condition (I.5) is a finite rank condition. One may use the recursive
word family
\[
M_{1,1}=A_1B,\quad M_{2,1}=A_2B,
\]
\[
M_{1,k+1}=(A_1M_{1,k},A_1M_{2,k}),\qquad
M_{2,k+1}=(A_2M_{1,k},A_2M_{2,k}),
\]
and stop as soon as the generated column space stabilizes.
\end{introremark}

\begin{introproposition}
\textbf{Finite rank test.}
Let
\[
\mathcal V_k
=\operatorname{span}\{w(A_1,A_2)B\eta:
|w|\le k,\ \eta\in\R^m\}.
\]
Then
\[
\mathcal V_{k+1}
=\mathcal V_k+A_1\mathcal V_k+A_2\mathcal V_k,
\qquad
\mathcal R(A_1,A_2,B)=\mathcal V_{n-1}.
\tag{I.8}
\]
Consequently, (I.5) is decidable by forming only the words of length at
most \(n-1\).
\end{introproposition}

\begin{proof}
The recurrence is immediate from the definition. If
\(\mathcal V_{k+1}=\mathcal V_k\), then \(\mathcal V_k\) is invariant under
\(A_1,A_2\), so all subsequent spaces equal \(\mathcal V_k\). Otherwise
the dimension increases by at least one. Starting from
\(\mathcal V_0=\operatorname{Ran}B\), stabilization therefore occurs no
later than \(k=n-1\).
\end{proof}

Section~1 proves well-posedness and develops the tools used later:
duality, factorization of the dual dynamics, the internal stochastic
Gramian, low-frequency observability, high-frequency dissipation, and the
spectral iteration. Section~2 proves both directions of the main theorem.
Section~3 examines pure drift transfer, pure martingale-coupling transfer,
the constant two-state case, and a genuinely mixed word. The final section
summarizes the result and discusses open extensions.

\section{Well-posedness and tools}

This section collects the analytic and algebraic ingredients used in the
proof: well-posedness, duality, factorization of the dual dynamics, the
internal stochastic Gramian, low-frequency observability, and high-frequency
dissipation.

\subsection{Well-posedness}

\begin{proposition}[Well-posedness]
For every
\[
y_T\in L^2_{\mathcal F_T}(\Omega;L^2(\D)^n),
\qquad
v\in L^2_{\F}(0,T;L^2(G)^m),
\]
system (I.1) has a unique variational adapted solution satisfying
\[
\begin{split}
y&\in L^2_{\F}\bigl(\Omega;C([0,T];L^2(\D)^n)\bigr)
 \cap L^2_{\F}(0,T;H_0^1(\D)^n),\\
Y&\in L^2_{\F}(0,T;L^2(\D)^n).
\end{split}
\]
Moreover,
\[
\E\sup_{0\le t\le T}\|y(t)\|^2
+\E\int_0^T\bigl(\|\nabla y(t)\|^2+\|Y(t)\|^2\bigr)dt
\le C_T\left(\E\|y_T\|^2+
\E\int_0^T\|v(t)\|_{L^2(G)^m}^2dt\right).
\]
\end{proposition}

\begin{proof}
This is the standard variational well-posedness estimate for a linear
backward stochastic parabolic equation; see, for
example,~\cite{TangZhang,LuZhang}. The term \(A_2Y\) is absorbed by Young's
inequality in the Itô energy estimate, followed by Gronwall's lemma.
\end{proof}

\subsection{Duality and the observability principle}

The dual forward system is
\begin{equation}
\begin{cases}
dz=(d\Delta z-A_1^*z)\,dt-A_2^*z\,dW_t&\text{in }Q,\\
z=0&\text{on }\Sigma,\\
z(0)=z_0&\text{in }\D.
\end{cases}
\tag{1.1}
\end{equation}

\begin{proposition}
System (I.1) is null controllable if and only if, for every
\[
z_0\in L^2(\D)^n,
\]
the corresponding solution of (1.1)
satisfies
\begin{equation}
\E\|z(T)\|_{L^2(\D)^n}^2
\le C_T\E\int_0^T\|B^*z(t)\|_{L^2(G)^m}^2\,dt.
\tag{1.2}
\end{equation}
\end{proposition}

\begin{proof}
Applying Itô's formula to the \(L^2(\D)^n\) inner product of a solution
\((y,Y)\) of (I.1) and a solution \(z\) of (1.1), all drift, diffusion, and
quadratic-covariation terms cancel except the control term. Thus
\[
\E\langle y_T,z(T)\rangle-\E\langle y(0),z_0\rangle
=\E\int_0^T\!\!\int_G\langle v,B^*z\rangle\,dx\,dt.
\]
Let
\[
\mathcal X_T=L^2_{\mathcal F_T}(\Omega;L^2(\D)^n),\quad
\mathcal X_0=L^2(\D)^n,\quad
\mathcal V=L^2_{\F}(0,T;L^2(G)^m).
\]
Denote by \(L:\mathcal X_T\to\mathcal X_0\) the terminal-to-initial map
for \(v=0\), and by \(M:\mathcal V\to\mathcal X_0\) the
control-to-initial map for zero terminal datum. Null controllability is
equivalent to
\(\operatorname{Ran}L\subset\operatorname{Ran}M\).
The identity above identifies
\[
L^*z_0=z(T),\qquad M^*z_0=-B^*z|_{G\times(0,T)};
\]
the harmless sign depends only on the convention used to write
\(L y_T+Mv=0\). Douglas' range inclusion theorem now shows that the range
inclusion is equivalent to
\(\|L^*z_0\|^2\le C_T\|M^*z_0\|^2\), which is (1.2).
Here we used the canonical identification
\(L^2_{\mathcal F_0}(\Omega;L^2(\D)^n)=L^2(\D)^n\), which follows from the
triviality of \(\mathcal F_0\). Random initial data are needed only in the
translated arguments on intervals \([s,s+\tau]\) with \(s>0\).
\end{proof}

\subsection{Factorization of the dual dynamics}

Let \(S(t)\) solve
\begin{equation}
dS(t)=-A_1^*S(t)\,dt-A_2^*S(t)\,dW_t,\qquad S(0)=I_n.
\tag{1.3}
\end{equation}
More generally, for \(0\le s\le t\), let \(S(t,s)\) be the unique matrix
solution of
\[
dS(t,s)=-A_1^*S(t,s)\,dt-A_2^*S(t,s)\,dW_t,
\qquad S(s,s)=I_n.
\]
Thus \(S(t)=S(t,0)\). Since the coefficients are deterministic and constant,
\(S(t,s)\) is measurable with respect to
\(\sigma(W_r-W_s:s\le r\le t)\), is independent of \(\mathcal F_s\), and has
the same law as \(S(t-s,0)\).

\begin{lemma}
The solution of (1.1) is
\begin{equation}
z(t)=S(t)e^{td\Delta}z_0.
\tag{1.4}
\end{equation}
\end{lemma}

\begin{proof}
First take \(z_0\) in a finite Dirichlet spectral subspace and set
\(h(t)=e^{td\Delta}z_0\). Then \(h\) has finite variation and
\(dh=d\Delta h\,dt\). The matrix \(S(t)\) is independent of the spatial
variable, so it commutes with \(\Delta\), and its quadratic covariation with
\(h\) is zero. Itô's product formula therefore gives
\[
\begin{split}
d(S(t)h(t))
&=(dS(t))h(t)+S(t)\,dh(t)\\
&=\bigl(d\Delta(S(t)h(t))-A_1^*S(t)h(t)\bigr)dt
   -A_2^*S(t)h(t)\,dW_t.
\end{split}
\]
Moreover, \(S(0)h(0)=z_0\), and the Dirichlet boundary condition is
preserved by multiplication with \(S(t)\). Hence \(S(t)h(t)\) solves (1.1),
and uniqueness proves (1.4) for such data. For general
\(z_0\in L^2(\D)^n\), approximate by the spectral
projections \(\Pi_\Lambda z_0\) and use the standard \(L^2\) stability of
the forward equation and the heat semigroup to pass to the limit.
\end{proof}

Define
\[
P(t)=\E[S(t)^*S(t)],\qquad
R(t)=\E[S(t)^*BB^*S(t)].
\tag{1.5}
\]
Writing \(h(t)=e^{td\Delta}z_0\), we have
\begin{align}
\E\|z(T)\|_{L^2(\D)^n}^2
&=\int_\D h(T,x)^*P(T)h(T,x)\,dx,
\tag{1.6}\\
\E\int_0^T\|B^*z(t)\|_{L^2(G)^m}^2dt
&=\int_0^T\!\!\int_G h(t,x)^*R(t)h(t,x)\,dx\,dt.
\tag{1.7}
\end{align}
Since the stochastic fundamental matrix is invertible almost surely,
\(P(T)\) is positive definite. Hence the quantity in (1.6) is equivalent
to \(\|h(T)\|_{L^2(\D)^n}^2\).
Indeed, with \(C_0=-A_1^*\) and \(C_1=-A_2^*\), the inverse flow solves
\[
d(S^{-1})=S^{-1}(-C_0+C_1^2)\,dt-S^{-1}C_1\,dW_t,
\qquad S^{-1}(0)=I_n,
\]
as follows directly from Itô's product formula.

\subsection{The internal stochastic Gramian and the word condition}

Set
\begin{equation}
G_\tau=\int_0^\tau R(t)\,dt
=\E\int_0^\tau S(t)^*BB^*S(t)\,dt.
\tag{1.8}
\end{equation}
Let
\[
C_0=-A_1^*,\qquad C_1=-A_2^*,
\]
and
\[
N=\bigcap_{w\in\mathcal W(C_0,C_1)}\Ker(B^*w).
\tag{1.9}
\]
Here \(\mathcal W(C_0,C_1)\) is the set of all finite words in
\(C_0,C_1\), including the empty word \(I_n\), with the same convention as
in (I.4).

\begin{lemma}[Kernel of the Gramian]
For every \(\tau>0\),
\begin{equation}
\Ker G_\tau=N.
\tag{1.10}
\end{equation}
Consequently, \(G_\tau\) is positive definite if and only if (I.5) holds.
\end{lemma}

\begin{proof}
If \(\xi\in N\), then \(N\) is invariant under \(C_0,C_1\). Hence
\(S(t)\xi\in N\) and \(B^*S(t)\xi=0\), so \(\xi\in\Ker G_\tau\).

Conversely, let \(\xi\in\Ker G_\tau\). Then
\[
B^*S(t)\xi=0
\]
for \(dt\otimes d\mathbb P\)-almost every \((t,\omega)\). Since this process
has continuous paths, it vanishes indistinguishably on \([0,\tau]\).
Suppose that
\(B^*wS(t)\xi=0\) for a word \(w\). Its stochastic differential is
\[
d(B^*wS(t)\xi)
=B^*wC_0S(t)\xi\,dt+B^*wC_1S(t)\xi\,dW_t.
\]
Uniqueness of the semimartingale decomposition implies that the two
coefficients vanish \(dt\otimes d\mathbb P\)-almost everywhere; their
continuity again gives indistinguishable vanishing. Induction over the word
length and evaluation at
\(t=0\) give \(B^*w\xi=0\) for every word \(w\). Thus
\(\Ker G_\tau=N\). Taking orthogonal complements gives the equivalence
with (I.5), because transposition reverses the order of a word and the
family of all words is stable under reversal; the signs in \(C_0,C_1\)
do not change the generated span.
\end{proof}

\begin{lemma}[Polynomial small-time coercivity]
Assume (I.5). There exist \(c>0\), \(\tau_0>0\), and an integer \(q\ge1\)
such that
\begin{equation}
G_\tau\ge c\tau^qI_n,\qquad 0<\tau\le\tau_0.
\tag{1.11}
\end{equation}
\end{lemma}

\begin{proof}
On the finite-dimensional space of symmetric matrices, define
\[
\mathcal L(Q)=C_0^*Q+QC_0+C_1^*QC_1.
\]
For a symmetric matrix \(Q\), Itô's formula gives
\[
d(S(t)^*QS(t))
=S(t)^*\mathcal L(Q)S(t)\,dt
+S(t)^*(C_1^*Q+QC_1)S(t)\,dW_t.
\]
The stochastic integral has zero expectation: its integrand is square
integrable on every bounded time interval, since the linear matrix SDE has
moments of every order. Define
\[
\mathcal T_tQ=\E[S(t)^*QS(t)].
\]
The flow identity \(S(t+s,0)=S(t+s,s)S(s,0)\), together with the
independence of \(S(t+s,s)\) from \(\mathcal F_s\), yields
\[
\mathcal T_{t+s}Q=\mathcal T_s(\mathcal T_tQ).
\]
Thus \((\mathcal T_t)_{t\ge0}\) is a continuous semigroup. Dividing the
integrated Itô identity by \(t\) and letting \(t\downarrow0\) shows that its
generator is \(\mathcal L\). Consequently, in this finite-dimensional
space,
\[
\E[S(t)^*QS(t)]=\mathcal T_tQ=e^{t\mathcal L}Q.
\]
Therefore
\[
G_\tau
=\sum_{k=0}^{\infty}
\frac{\tau^{k+1}}{(k+1)!}\mathcal L^k(BB^*),
\]
so \(G_\tau\) is analytic in \(\tau\). By the preceding lemma,
\(\det G_\tau>0\) for every \(\tau>0\). Thus
\[
\det G_\tau=a_r\tau^r+O(\tau^{r+1})
\]
for some \(r\ge n\) and \(a_r>0\). Moreover,
\(\|G_\tau\|\le C\tau\). If
\(\lambda_1(\tau)\le\cdots\le\lambda_n(\tau)\) are its eigenvalues,
\[
\lambda_1(\tau)
\ge\frac{\det G_\tau}{\|G_\tau\|^{n-1}}
\ge c\tau^{r-n+1}.
\]
This proves (1.11).
\end{proof}

\begin{remark}[An effective admissible exponent]
The proof gives a directly computable exponent. Put
\[
H_j=\frac1{j!}\mathcal L^{j-1}(BB^*),\qquad j\ge1,
\]
so that \(G_\tau=\sum_{j\ge1}\tau^jH_j\), and let
\[
r_*=\min\left\{r\ge n:
[\tau^r]\det\!\left(\sum_{j=1}^{r}\tau^jH_j\right)\ne0\right\}.
\tag{1.12}
\]
Here \([\tau^r]\) denotes coefficient extraction. Under (I.5), the set in
(1.12) is nonempty, its first coefficient is positive, and
\[
q_{\rm det}=r_*-n+1
\tag{1.13}
\]
is an admissible exponent in (1.11). Thus, when the matrix entries are
specified exactly, \(q_{\rm det}\) can be found by symbolic or exact
matrix arithmetic, successively computing
\(\mathcal L^j(BB^*)\) and the determinant coefficients. This exponent
need not be the smallest possible one: the determinant argument bounds the
other \(n-1\) eigenvalues only by \(O(\tau)\). No optimality of
\(q_{\rm det}\) is used below.

The ordinary length of the shortest spanning words does not by itself
determine the small-time exponent. Under Brownian scaling, a drift letter
and a noise letter have different temporal weights. The two-state examples
below already show this: both transfers use a word of length one, whereas
the drift and noise Gramians have respective coercivity orders three and
two. Thus the quantitative argument must use formula (1.12), not ordinary
word length.
\end{remark}

\subsection{Low-frequency observability}

Let \((\lambda_k,\phi_k)\) be the Dirichlet eigenpairs of \(-\Delta\), and
let \(\Pi_\Lambda\) be the orthogonal projection on
\[
E_\Lambda^n
=\operatorname{span}\{\phi_k:\lambda_k\le\Lambda\}^n.
\]

\begin{proposition}
Assume (I.5). There exist \(C,q,\tau_0>0\) such that, for
\(z_0\in E_\Lambda^n\) and \(0<\tau\le\tau_0\),
\begin{equation}
\|z_0\|_{L^2(\D)^n}^2
\le C\tau^{-q}e^{C\sqrt\Lambda+2d\Lambda\tau}
\E\int_0^\tau\|B^*z(t)\|_{L^2(G)^m}^2dt.
\tag{1.14}
\end{equation}
\end{proposition}

\begin{proof}
Write \(z_0=\sum_{\lambda_k\le\Lambda}\xi_k\phi_k\). Then
\[
z(t)=\sum_{\lambda_k\le\Lambda}
e^{-d\lambda_kt}S(t)\xi_k\phi_k.
\]
We first integrate over the whole domain, so no pointwise-in-\(x\) coercivity
is used. Orthogonality of the eigenfunctions gives
\[
\begin{split}
\E\int_0^\tau\|B^*z(t)\|_{L^2(\D)^m}^2dt
&=\sum_{\lambda_k\le\Lambda}
\int_0^\tau e^{-2d\lambda_kt}\,
\xi_k^*R(t)\xi_k\,dt\\
&\ge e^{-2d\Lambda\tau}
\sum_{\lambda_k\le\Lambda}\xi_k^*G_\tau\xi_k.
\end{split}
\]
Consequently, (1.11) gives
\[
\|z_0\|^2
\le C\tau^{-q}e^{2d\Lambda\tau}
\E\int_0^\tau\|B^*z(t)\|_{L^2(\D)^m}^2dt.
\]
For each \((t,\omega)\), \(B^*z(t,\omega)\in E_\Lambda^m\). The
Lebeau--Robbiano spectral inequality in the Dirichlet-domain setting
\cite[Theorem~14.6; chapter pp.~223--239]{JerisonLebeau}; see also the
foundational heat-control result~\cite{LebeauRobbiano}.
\[
\|f\|_{L^2(\D)^m}^2
\le Ce^{C\sqrt\Lambda}\|f\|_{L^2(G)^m}^2
\]
completes the proof.
\end{proof}

\begin{corollary}[Translated and conditional estimate]
Let \(s\ge0\), \(0<\tau\le\tau_0\), and let
\[
 z_s\in L^2_{\mathcal F_s}(\Omega;E_\Lambda^n).
\]
If \(z\) solves the dual equation on \([s,s+\tau]\) with \(z(s)=z_s\),
then
\begin{equation}
\E\|z_s\|^2
\le C\tau^{-q}e^{C\sqrt\Lambda+2d\Lambda\tau}
\E\int_s^{s+\tau}\|B^*z(t)\|_{L^2(G)^m}^2dt,
\tag{1.15}
\end{equation}
with constants independent of \(s\).
\end{corollary}

\begin{proof}
Write
\(z_s=\sum_{\lambda_k\le\Lambda}\xi_k\phi_k\), where every
\(\xi_k\) is \(\mathcal F_s\)-measurable and square integrable. For
\(0\le r\le\tau\), the factorization on \([s,s+\tau]\) gives
\[
z(s+r)=\sum_{\lambda_k\le\Lambda}
e^{-d\lambda_kr}S(s+r,s)\xi_k\phi_k.
\]
By the independence and stationarity properties stated after (1.3),
orthogonality of the \(\phi_k\)'s and conditional Tonelli's theorem yield,
almost surely,
\[
\begin{split}
&\E\left(\left.
 \int_s^{s+\tau}\!\|B^*z(t)\|_{L^2(\D)^m}^2dt
 \right|\mathcal F_s\right)\\
&\quad=\sum_{\lambda_k\le\Lambda}\int_0^\tau
 e^{-2d\lambda_kr}\xi_k^*R(r)\xi_k\,dr\\
&\quad\ge e^{-2d\Lambda\tau}
 \sum_{\lambda_k\le\Lambda}\xi_k^*G_\tau\xi_k
 \ge c\tau^qe^{-2d\Lambda\tau}\|z_s\|^2.
\end{split}
\]
After taking expectation, the deterministic spectral inequality, applied
to \(B^*z(t,\omega)\in E_\Lambda^m\) for
\(dt\otimes d\mathbb P\)-almost every \((t,\omega)\), replaces the
whole-domain observation by the observation on \(G\) and proves (1.15).
\end{proof}

\begin{lemma}[Partial spectral controllability]
Let \(0\le a<b\), \(\tau=b-a\le\tau_0\), and let
\(\eta\in L^2_{\mathcal F_b}(\Omega;L^2(\D)^n)\). Under (I.5), there is an
adapted control \(v_\Lambda\), supported in \(G\times(a,b)\), such that the
solution of the backward equation with terminal datum \(y(b)=\eta\)
satisfies
\[
\Pi_\Lambda y(a)=0
\]
and
\begin{equation}
\E\int_a^b\|v_\Lambda(t)\|_{L^2(G)^m}^2dt
\le C\tau^{-q}e^{C\sqrt\Lambda+2d\Lambda\tau+C\tau}\E\|\eta\|^2.
\tag{1.16}
\end{equation}
The same right-hand side, up to a fixed multiplicative constant, bounds
\(\E\|y(a)\|^2\).
\end{lemma}

\begin{proof}
Set
\[
\mathcal H_\Lambda
=L^2_{\mathcal F_a}(\Omega;E_\Lambda^n),\qquad
\mathcal Z_{a,b}=L^2_\F(a,b;L^2(G)^m),
\]
and define
\[
\mathcal U_{b,a}z_a=z(b),\qquad
\mathcal O_{b,a}z_a=B^*z|_{G\times(a,b)}\in\mathcal Z_{a,b}.
\]
The norm in \(\mathcal Z_{a,b}\) is precisely
\[
\|\mathcal O_{b,a}z_a\|_{\mathcal Z_{a,b}}^2
=\E\int_a^b\|B^*z(t)\|_{L^2(G)^m}^2\,dt,
\]
so the translated estimate (1.15), with \(s=a\), says
\[
\|z_a\|_{\mathcal H_\Lambda}^2
\le K(\Lambda,\tau)\|\mathcal O_{b,a}z_a\|_{\mathcal Z_{a,b}}^2,
\qquad
K(\Lambda,\tau)
\le C\tau^{-q}e^{C\sqrt\Lambda+2d\Lambda\tau}.
\tag{1.17}
\]
Since the dual evolution preserves \(E_\Lambda^n\), its energy at \(b\)
is bounded by \(Ce^{C\tau}\) times its energy at \(a\). Therefore the
functional
\[
\ell_\eta(z_a)=\E\langle\eta,\mathcal U_{b,a}z_a\rangle
\]
satisfies, by Cauchy--Schwarz, the energy estimate, and (1.17),
\[
\begin{split}
|\ell_\eta(z_a)|
&\le \|\eta\|_{L^2_{\mathcal F_b}(\Omega;L^2(\D)^n)}
\|\mathcal U_{b,a}z_a\|_{L^2_{\mathcal F_b}(\Omega;L^2(\D)^n)}\\
&\le Ce^{C\tau/2}
\|\eta\|_{L^2_{\mathcal F_b}(\Omega;L^2(\D)^n)}
\|z_a\|_{\mathcal H_\Lambda}\\
&\le Ce^{C\tau/2}K(\Lambda,\tau)^{1/2}
\|\eta\|_{L^2_{\mathcal F_b}(\Omega;L^2(\D)^n)}
\|\mathcal O_{b,a}z_a\|_{\mathcal Z_{a,b}}.
\end{split}
\]
In particular, if \(\mathcal O_{b,a}z_a=0\), then
\(\ell_\eta(z_a)=0\). Hence
\(\mathcal O_{b,a}z_a\mapsto\ell_\eta(z_a)\) is a well-defined bounded
functional on \(\operatorname{Ran}\mathcal O_{b,a}\subset\mathcal Z_{a,b}\).
By Hahn--Banach and the Riesz
representation theorem, there exists \(v_\Lambda\) such that
\[
\ell_\eta(z_a)
=\E\int_a^b\!\!\int_G
\langle v_\Lambda,B^*z\rangle\,dx\,dt
\quad\text{for every }z_a\in\mathcal H_\Lambda,
\]
and
\[
\|v_\Lambda\|_{\mathcal Z_{a,b}}^2
\le Ce^{C\tau}K(\Lambda,\tau)\E\|\eta\|^2.
\]
This is (1.16). Itô's formula on \([a,b]\), applied to
the controlled backward solution \(y\) with \(y(b)=\eta\) and to the
forward solution \(z=\mathcal U_{\cdot,a}z_a\), gives
\[
\E\langle\eta,\mathcal U_{b,a}z_a\rangle
-\E\langle y(a),z_a\rangle
=\E\int_a^b\!\!\int_G
\langle v_\Lambda,B^*z\rangle\,dx\,dt.
\tag{1.18}
\]
Comparison with the Riesz representation of \(\ell_\eta\) gives
\(\E\langle y(a),z_a\rangle=0\) for every
\(z_a\in\mathcal H_\Lambda\). Since
\(\Pi_\Lambda y(a)\in\mathcal H_\Lambda\), one may choose
\(z_a=\Pi_\Lambda y(a)\); hence
\(\E\|\Pi_\Lambda y(a)\|^2=0\), which proves
\(\Pi_\Lambda y(a)=0\) almost surely. Finally, the standard \(L^2\) estimate for the
linear backward equation, together with (1.16), gives the state estimate.
\end{proof}

\subsection{High-frequency dissipation and spectral iteration}

\begin{lemma}[Dissipation]
There exist structural constants \(K_{\rm dis}\ge1\) and
\(\omega_{\rm dis}\ge0\) such that
\begin{equation}
\E\|(I-\Pi_\Lambda)z(t)\|^2
\le K_{\rm dis}e^{\omega_{\rm dis}(t-s)-2d\Lambda(t-s)}
\E\|(I-\Pi_\Lambda)z(s)\|^2
\tag{1.19}
\end{equation}
for \(0\le s<t\le T\).
The adjoint backward evolution satisfies the same estimate from its right
endpoint to its left endpoint when the control vanishes and the terminal
datum belongs to \((I-\Pi_\Lambda)L^2(\D)^n\).
\end{lemma}

\begin{proof}
Put \(r=t-s\). By factorization,
\[
(I-\Pi_\Lambda)z(t)
=S(t,s)e^{rd\Delta}(I-\Pi_\Lambda)z(s).
\]
The increment \(S(t,s)\) is independent of \(\mathcal F_s\). Hence,
conditionally on \(\mathcal F_s\),
\[
\begin{split}
\E\bigl(\|(I-\Pi_\Lambda)z(t)\|^2\mid\mathcal F_s\bigr)
&\le \|\E[S(t,s)^*S(t,s)]\|
 \|e^{rd\Delta}(I-\Pi_\Lambda)z(s)\|^2\\
&\le K_{\rm dis}e^{\omega_{\rm dis}r-2d\Lambda r}
 \|(I-\Pi_\Lambda)z(s)\|^2,
\end{split}
\]
where the first exponential bound follows from the standard moment estimate
for the linear matrix SDE. Taking expectation proves the forward assertion.

Let
\[
\mathcal X_s=L^2_{\mathcal F_s}(\Omega;L^2(\D)^n),\qquad
\mathcal X_t=L^2_{\mathcal F_t}(\Omega;L^2(\D)^n),
\]
and let \(U(t,s):\mathcal X_s\to\mathcal X_t\) be the forward solution
operator. If \(\eta\in\mathcal X_t\) and \((y,Y)\) is the free backward
solution on \([s,t]\) with \(y(t)=\eta\), Itô's formula gives
\[
\E\langle\eta,U(t,s)z_s\rangle
=\E\langle y(s),z_s\rangle
\quad\text{for every }z_s\in\mathcal X_s.
\]
Thus the terminal-to-left-endpoint map is exactly \(U(t,s)^*\).
All coefficients commute with \(\Pi_\Lambda\), so the high-frequency
subspaces are invariant for both maps. Restricting the forward estimate to
\((I-\Pi_\Lambda)\mathcal X_s\) transfers it to the backward map with the
same operator norm. Squaring that norm gives the
backward assertion in (1.19).
\end{proof}

\begin{lemma}[Spectral iteration]
Let \(U(t,s)\) be the forward evolution of an adapted dual/control pair,
acting between the corresponding adapted \(L^2\) spaces, and suppose it
commutes with every \(\Pi_\Lambda\). Assume that the pair satisfies the
duality identity (1.18) on every subinterval and the following uniform
partial-control property: for every \(0\le a<b\) with \(b-a\le\tau_0\),
every terminal datum \(\eta\) at \(b\), and every \(\Lambda\), one can choose a control on
\((a,b)\) that makes \(\Pi_\Lambda y(a)=0\), with its squared norm and the
mean squared state at \(a\) bounded by
\(C K(\Lambda,b-a)\mathbb E\|\eta\|^2\). Assume the cost obeys
\begin{equation}
 K(\Lambda,\tau)\le C\tau^{-q}e^{C\sqrt\Lambda+2d\Lambda\tau},
\tag{1.20}
\end{equation}
uniformly with respect to the initial time, and that the complementary
frequencies satisfy (1.19). Then the controlled evolution is null
controllable on every interval of positive length. Equivalently, \(U\)
satisfies final-state observability.
\end{lemma}

\begin{proof}
We give the iteration in detail because the observation matrix depends on
time after taking expectation. The partial-control hypothesis gives a
control on an interval of length
\(\tau/4\), namely the rightmost quarter of the current interval, which
annihilates the modes below \(\Lambda\) at the left endpoint of that
quarter (the controlled equation is read backward), with squared control
cost and squared state amplification bounded by
\[
C\tau^{-q}e^{C\sqrt\Lambda+\frac d2\Lambda\tau+C\tau}.
\tag{1.21}
\]
Denote by \(\eta\) the state at the right endpoint, by \(\eta_c\) the
state after the controlled quarter, and by \(\eta_{\rm out}\) the state at
the left endpoint of the full interval. Then
\(\Pi_\Lambda\eta_c=0\), while (1.16) and the state estimate give
\[
\E\|\eta_c\|^2
\le C\tau^{-q}
e^{C\sqrt\Lambda+d\Lambda\tau/2+C\tau}\E\|\eta\|^2.
\]
On the remaining left interval, of length \(3\tau/4\), the control is set
equal to zero. Since \(\eta_c\) has only frequencies above \(\Lambda\),
the backward counterpart of (1.19) gives
\[
\E\|\eta_{\rm out}\|^2
\le Ce^{C\tau-3d\Lambda\tau/2}\E\|\eta_c\|^2.
\]
Multiplying these two inequalities yields, after enlarging \(C\),
\begin{equation}
 \mathcal E_{\rm out}
 \le C\tau^{-q}
 \exp\!\left(C\sqrt\Lambda+C\tau-d\Lambda\tau\right)
 \mathcal E_{\rm in}.
\tag{1.22}
\end{equation}
Here \(\mathcal E_{\rm in}\) is the mean squared norm at the right endpoint
and \(\mathcal E_{\rm out}\) that at the left endpoint. This argument is
valid for random
endpoint data, since all the duality spaces are the adapted \(L^2\) spaces
on the corresponding subinterval.

It suffices first to work on an interval of length
\[
0<T_*\le\min\{T,1,2\tau_0\}.
\]
Controllability on this subinterval gives controllability on the whole
interval: set the control equal to zero on \((T_*,T)\), solve there first,
and use the resulting \(\mathcal F_{T_*}\)-measurable state as terminal datum
for the controlled equation on \((0,T_*)\).
Put
\[
 \tau_j=T_*2^{-j},\qquad
 \Lambda_j=\beta T_*^{-2}4^j,\qquad
 T_j=\sum_{k=1}^j\tau_k,\qquad
 I_j=(T_*-T_j,T_*-T_{j-1}),\quad T_0=0.
\tag{1.23}
\]
Thus \(T_j\uparrow T_*\). The intervals \(I_j\) are read from the terminal
time toward \(0\), as required by the backward controlled equation, and
every controlled quarter has length at most \(\tau_0\). Write
\[
\mathcal E_0=\E\|y(T_*)\|^2,\qquad
\mathcal E_j=\E\|y(T_*-T_j)\|^2\quad(j\ge1),
\]
where \(y\) is the trajectory constructed by the successive partial
controls and free backward propagation. Applying (1.22) successively on
\(I_j\) gives
\[
 \mathcal E_j
 \le \prod_{k=1}^j
 \left[
 CT_*^{-q}2^{qk}
 \exp\!\left\{
 \frac{\sqrt\beta}{T_*}C2^k
 -\frac{d\beta}{T_*}2^k+C\tau_k
 \right\}\right]\mathcal E_0.
\tag{1.24}
\]
Since \(0<T_*\le1\) and \(k\le2^k\), there is a constant \(M>0\),
independent of \(T_*,k,\beta\), such that all logarithms of the polynomial
prefactors in the \(k\)-th step are bounded by \(M2^k/T_*\). Set
\[
\alpha_\beta=d\beta-C\sqrt\beta-M,\qquad
\rho_\beta=C\sqrt\beta+\frac d2\beta+M.
\]
Choose \(\beta\) so large that
\[
\alpha_\beta>\rho_\beta>0;
\tag{1.25}
\]
this is possible because
\(\alpha_\beta-\rho_\beta
=\frac d2\beta-2C\sqrt\beta-2M\).
It follows from (1.24) and
\(\sum_{k=1}^j2^k=2^{j+1}-2\) that
\[
 \mathcal E_j
 \le C\exp\!\left(\frac{C}{T_*}\right)
 \exp\!\left(-\frac{\alpha_\beta}{T_*}2^{j+1}\right)
 \mathcal E_0\longrightarrow0.
\tag{1.26}
\]
Indeed, the sum of the terms \(C\tau_k\) is bounded by \(CT_*\), while
the missing constant \(2\alpha_\beta/T_*\) comes from
\(\sum_{k=1}^j2^k=2^{j+1}-2\). Thus the constant displayed in (1.26) is
uniform in \(j\) and has the required small-time form.

For \(j\ge2\), apply (1.26) with index \(j-1\) to bound the energy
entering the \(j\)-th interval. The control cost from (1.21) then gives
\[
 \|v_j\|_{L^2_\F(I_j;L^2(G)^m)}^2
 \le C\exp\!\left(\frac{C}{T_*}\right)
 \exp\!\left(\frac{\rho_\beta}{T_*}2^j\right)\mathcal E_{j-1}
 \le C\exp\!\left(\frac{C}{T_*}\right)
 \exp\!\left[-\frac{\alpha_\beta-\rho_\beta}{T_*}2^j\right]
 \mathcal E_0,
\tag{1.27}
\]
where (1.25) was used in the last inequality. Hence
\[
\sum_{j\ge1}
\exp\!\left[-\frac{\alpha_\beta-\rho_\beta}{T_*}2^j\right]
\le
\sum_{j\ge1}e^{-(\alpha_\beta-\rho_\beta)2^j}<\infty,
\]
because \(T_*\le1\). Therefore
\(\sum_j\|v_j\|_{L^2_\F(I_j;L^2(G)^m)}^2<\infty\). Gluing the controls on the disjoint
intervals \(I_j\) produces an adapted \(L^2\) control. Let
\(t_j=T_*-T_j\downarrow0\) be the left endpoints of the intervals
already treated. At each \(t_j\), the constructed state has mean squared
norm \(\mathcal E_j\), which tends to zero by (1.26). By the well-posedness
estimate, the corresponding backward solution belongs to
\(L^2_\F(\Omega;C([0,T_*];L^2(\D)^n))\); hence its state at \(0\) is the
limit in \(L^2(\Omega;L^2(\D)^n)\) of its states at \(t_j\). It follows
that \(\E\|y(0)\|^2=0\).
For \(j=1\), the partial-control bound is applied directly with
\(\mathcal E_0\); its prefactor is absorbed into the same
\(C\exp(C/T_*)\) bound after increasing \(C\).
Consequently, for \(0<T_*\le1\),
\[
\sum_{j\ge1}\|v_j\|_{L^2_\F(I_j;L^2(G)^m)}^2
\le C e^{C/T_*}\mathcal E_0.
\tag{1.28}
\]
For the original interval, take
\(T_*=\min\{T,1,2\tau_0\}\). The free backward estimate on
\((T_*,T)\) bounds \(\mathcal E_0\) by
\(Ce^{C(T-T_*)}\E\|y_T\|^2\). Thus (1.28) gives a finite control cost
for every fixed \(T>0\), with a constant allowed to depend on \(T\) and
the structural data. If \(0<T\le\min\{1,2\tau_0\}\), then \(T_*=T\)
and (1.28) directly gives the claimed \(Ce^{C/T}\) bound. If
\(2\tau_0<T\le1\), the fixed positive cutoff \(T_*=2\tau_0\) gives a
uniform finite bound on that remaining range; enlarging \(C\), depending
on \(\tau_0\), puts it in the form \(Ce^{C/T}\) for all \(0<T\le1\).
Applying the iteration to every terminal datum gives
\(\operatorname{Ran}L\subset\operatorname{Ran}M\) for the terminal-to-
initial and control-to-initial maps on this horizon. The Hilbert-space
duality principle (Douglas' lemma, as established above) then gives the
announced final-state observability inequality.
This is the usual stochastic Lebeau--Robbiano
construction~\cite{LebeauRobbiano,Lu2011,Liu2014,LiuLiu}; the explicit
choices (1.23) show that its constants are compatible with the polynomial
Gramian loss in (1.20).
\end{proof}

\section{Proof of the main theorem}

\subsection{Sufficiency of the word-rank condition}

\begin{proposition}[Observability]
If (I.5) holds, then (1.2) holds.
\end{proposition}

\begin{proof}
The translated estimate (1.15), the partial controllability lemma, and the
dissipation estimate (1.19) verify the hypotheses of the spectral iteration
lemma. Hilbert-space duality then yields the final-state observability
inequality (1.2).
\end{proof}

\subsection{Necessity of the word-rank condition}

The two descriptions of the unobservable space used above agree. Indeed,
for \(C_i=-A_i^*\), transposition reverses every word and multiplication by
the signs \((-1)^{|w|}\) does not affect its kernel. Since the collection of
all words is invariant under reversal,
\[
\bigcap_{w\in\mathcal W(C_0,C_1)}\Ker(B^*w)
=\mathcal R(A_1,A_2,B)^\perp.
\tag{2.1}
\]
Thus the symbol \(N\) in (1.9) has precisely this second description.

Assume (I.5) fails. Then \(N\) contains a nonzero vector \(\xi\) and is
invariant under \(A_1^*,A_2^*\).
Let \(z_0(x)=\phi_1(x)\xi\). Formula (1.4) shows that
\[
B^*z(t,x)=0
\]
almost surely in \(Q\), while \(z(T)\ne0\). Hence (1.2) fails, proving the
necessity in the main theorem.

Combining this necessity with the observability proposition and the
duality proposition proves the main theorem.

\section{Examples and comparison with classical Kalman conditions}

\begin{example}[Pure drift transfer and its small-time cost]
Let \(n=2\), \(B=e_1\), \(A_2=0\), and
\[
A_1=\begin{pmatrix}0&0\\a&0\end{pmatrix},\qquad a\ne0.
\]
Then \(\rank[B,A_1B]=2\). The internal dual flow is deterministic and
\[
S(t)=
\begin{pmatrix}1&-at\\0&1\end{pmatrix}.
\]
Consequently,
\[
G_\tau=
\begin{pmatrix}
\tau&-\dfrac{a\tau^2}{2}\\[0.4em]
-\dfrac{a\tau^2}{2}&\dfrac{a^2\tau^3}{3}
\end{pmatrix},
\qquad
\det G_\tau=\frac{a^2\tau^4}{12}.
\]
In particular,
\[
\lambda_{\min}(G_\tau)\asymp\tau^3
\quad\text{as }\tau\downarrow0.
\]
Here \(r_*=4\), so (1.13) gives \(q_{\rm det}=3\), which is optimal in
this example.
This recovers the observation transfer generated by a drift cascade.
\end{example}

\begin{example}[Pure martingale-coupling transfer]
Let \(n=2\), \(B=e_1\), \(A_1=0\), and
\[
A_2=\begin{pmatrix}0&0\\a&0\end{pmatrix},\qquad a\ne0.
\]
Again \(\rank[B,A_2B]=2\), but now the transfer is entirely stochastic.
Since \((A_2^*)^2=0\),
\[
S(t)=
\begin{pmatrix}1&-aW_t\\0&1\end{pmatrix}.
\]
It follows that
\[
R(t)=
\begin{pmatrix}1&0\\0&a^2t\end{pmatrix},
\qquad
G_\tau=
\begin{pmatrix}\tau&0\\0&\dfrac{a^2\tau^2}{2}\end{pmatrix}.
\]
Thus
\[
\lambda_{\min}(G_\tau)\asymp\tau^2.
\]
Here \(r_*=3\), and the determinant procedure gives the optimal exponent
\(q_{\rm det}=2\).
This example is not covered by a condition requiring a nonzero product of
two entries in the martingale-coupling matrix. It also shows that the
stochastic transfer has a different small-time order from the drift
transfer.
\end{example}

\begin{proposition}[The constant two-state case]
Let \(n=2\), \(m=1\), \(B=e_1\), and
\[
A_1=\begin{pmatrix}a_1&0\\0&a_2\end{pmatrix},
\qquad
A_2=\begin{pmatrix}d_{11}&d_{12}\\d_{21}&d_{22}\end{pmatrix}.
\]
Then condition (I.5) holds if and only if
\[
d_{21}\ne0.
\tag{3.1}
\]
In particular, no condition on \(d_{22}\) is needed.
\end{proposition}

\begin{proof}
If \(d_{21}\ne0\), then
\[
\det[B,A_2B]
=\det\begin{pmatrix}1&d_{11}\\0&d_{21}\end{pmatrix}
=d_{21}\ne0.
\]
Conversely, if \(d_{21}=0\), the line
\(\operatorname{span}\{e_1\}\) is invariant under both \(A_1\) and \(A_2\).
Every word applied to \(B\) therefore stays on that line, so (I.5)
fails.
\end{proof}

\begin{remark}
This statement should be compared with diffusion-coupling results based on
a localized weighted Itô identity. Such results can allow
\(d_{ij}=d_{ij}(t,x)\), but typically impose a signed nondegeneracy
condition involving \(d_{21}d_{22}\); see~\cite{FadiliTwoState}. The two conclusions are
complementary: spatially dependent coefficients lie outside the
factorization used here, while for constant matrices the exact word-rank
criterion shows that \(d_{22}\) is immaterial.
\end{remark}

\begin{example}[A genuinely mixed word]
Let \(B=e_1\) and
\[
A_1=
\begin{pmatrix}
1&0&0\\
0&0&0\\
0&1&0
\end{pmatrix},
\qquad
A_2=
\begin{pmatrix}
0&0&0\\
1&0&0\\
0&0&0
\end{pmatrix}.
\]
Then
\[
A_1B=e_1,\qquad A_2B=e_2,\qquad A_1A_2B=e_3.
\]
Consequently,
\[
\operatorname{span}\{B,A_1^kB,A_2^kB:k\ge1\}
=\operatorname{span}\{e_1,e_2\}\ne\R^3,
\]
whereas
\[
\det[B,A_2B,A_1A_2B]=1.
\]
Thus neither single-matrix Kalman family, nor the union of the two
pure-power families, is sufficient; the mixed word \(A_1A_2\) is genuinely
needed. Notice also that
\(A_1A_2B=e_3\ne e_2=A_2A_1B\), so the order of the letters matters.

This example also illustrates the possible loss in (1.13). Exact expansion
of the moment formula gives
\[
\det G_\tau=\frac{\tau^7}{72}+O(\tau^8),
\qquad q_{\rm det}=7-3+1=5.
\]
On the other hand, the leading \(2\times2\) block on
\(\operatorname{span}\{e_2,e_3\}\) has entries
\[
\frac{\tau^2}{2}+O(\tau^3),\qquad
-\frac{\tau^3}{6}+O(\tau^4),\qquad
\frac{\tau^4}{12}+O(\tau^5).
\]
Its Schur complement is
\[
\left(\frac1{12}-\frac{(1/6)^2}{1/2}\right)\tau^4
+O(\tau^5)
=\frac{\tau^4}{36}+O(\tau^5).
\]
The \(e_1\) direction is decoupled and its Gramian eigenvalue is
\(\tau+O(\tau^2)\), while the larger eigenvalue of the displayed block is
\(\tau^2/2+O(\tau^3)\).
Consequently,
\(\lambda_{\min}(G_\tau)\sim\tau^4/36\), whereas the general determinant
argument only supplies the weaker admissible exponent \(q_{\rm det}=5\).
\end{example}

\section*{Conclusion and perspectives}

Null controllability of the backward stochastic parabolic system with
constant coupling matrices and common scalar diffusion is equivalent to the
noncommutative word-rank condition generated by \(A_1\), \(A_2\), and
\(B\). The proof links the exact kernel description and small-time
coercivity of the internal stochastic Gramian to a Lebeau--Robbiano spectral
iteration. The examples show why genuinely mixed words may be indispensable
and why a classical Kalman condition based on a single matrix is generally
insufficient.

The factorization (1.4) relies on the scalar diffusion assumption
\(D=dI_n\). For a
general symmetric positive definite matrix \(D\), the \(k\)-th Laplace mode
contains \(A_1+\lambda_kD\). This suggests the modal family
\[
(A_1+\lambda_kD,A_2,B).
\]
Full parabolic controllability would require the corresponding word-rank
condition for every \(k\) and coercivity estimates uniform enough as
\(\lambda_k\to\infty\). Whether those modal conditions are
sufficient is left open.
Other directions include multidimensional Brownian noise, random or
time-dependent coupling matrices, and semilinear perturbations.

\section*{Declarations}

\textbf{Funding.}
This research received no specific grant from any funding agency in the
public, commercial, or not-for-profit sectors.

\textbf{Competing interests.}
The author declares that there are no competing interests related to this
work.

\textbf{Data availability.}
No research data were created or analyzed in this study. The short
exact-arithmetic scripts used to verify the examples and the joint
word-rank condition are available from the corresponding author on
reasonable request.

\textbf{Ethics approval and consent to participate.}
Not applicable.

\textbf{Declaration on the use of generative AI and AI-assisted tools.}
This text was entirely written by humans (TCH). AI was used as a tool for research, fact-checking, and proofreading.

\end{document}